\documentclass[12pt]{amsart}
\usepackage{amssymb}

\usepackage{enumerate}

\newcommand{\FF}{{\mathbb{F}}}
\newcommand{\QQ}{{\mathbb{Q}}}
\newcommand{\RR}{{\mathbb{R}}}

\newcommand{\bi}{{\mathbf{i}}}
\newcommand{\bE}{{\mathbf{E}}}
\newcommand{\bF}{{\mathbf{F}}}
\newcommand{\bG}{{\mathbf{G}}}
\newcommand{\bH}{{\mathbf{H}}}
\newcommand{\bL}{{\mathbf{L}}}
\newcommand{\bT}{{\mathbf{T}}}
\newcommand{\bU}{{\mathbf{U}}}
\newcommand{\bZ}{{\mathbf{Z}}}

\newcommand{\cE}{{\mathcal{E}}}

\newcommand{\fA}{{\mathfrak{A}}}
\newcommand{\fS}{{\mathfrak{S}}}

\newcommand{\ab}{{\operatorname{ab}}}
\newcommand{\Aut}{{\operatorname{Aut}}}
\newcommand{\Gal}{{\operatorname{Gal}}}
\newcommand{\IBr}{{\operatorname{IBr}}}
\newcommand{\Irr}{\operatorname{Irr}}
\newcommand{\Out}{\operatorname{Out}}

\newcommand{\Syl}{\operatorname{Syl}}
\newcommand{\Chevie}{{\sf{Chevie}}}

\newcommand{\GL}{\operatorname{GL}}
\newcommand{\SL}{\operatorname{SL}}
\newcommand{\PSL}{\operatorname{L}}
\newcommand{\PGL}{\operatorname{PGL}}

\newcommand{\SU}{\operatorname{SU}}
\newcommand{\PSU}{\operatorname{U}}
\newcommand{\GO}{\operatorname{GO}}
\newcommand{\OO}{\operatorname{O}}

\newcommand{\SO}{\operatorname{SO}}

\newcommand{\tw}[1]{{}^{#1}\!}
\def\pmod#1{~({\rm mod}~#1)}
\newcommand{\tchi}{{\tilde\chi}}

\def\irr#1{{\rm Irr}(#1)}
\def\irro#1{{\rm Irr}_{0}(#1)}
\def\irrs#1{{\rm Irr}_{\sigma}(#1)}
\def\cent#1#2{{\bf C}_{#1}(#2)}
\def\norm#1#2{{\bf N}_{#1}(#2)}
\def\zent#1{{\bf Z}{(#1)}}

\def\Oh#1#2{{\bf O}^{#1}(#2)}
\def\oh#1#2{{\bf O}_{#1}(#2)}

\let\la=\lambda
\let\vhi=\varphi
\let\sbs=\subseteq
\let\surj=\twoheadrightarrow

\newtheorem{thm}{Theorem}[section]
\newtheorem{lem}[thm]{Lemma}
\newtheorem{prop}[thm]{Proposition}
\newtheorem{cor}[thm]{Corollary}

\newtheorem*{thmA}{Theorem A}
\newtheorem*{thmB}{Theorem B}
\newtheorem*{thmC}{Theorem C}

\theoremstyle{definition}
\newtheorem{exmp}[thm]{Example}
\newtheorem{rem}[thm]{Remark}

\begin{document}

\title{Height Zero Conjecture with Galois Automorphisms}

\author{Gunter Malle}
\address{FB Mathematik, TU Kaiserslautern, Postfach 3049,
         67653 Kaisers\-lautern, Germany.}
\email{malle@mathematik.uni-kl.de}

\author{Gabriel Navarro}
\address{Departament of Mathematics, Universitat de Val\`encia, 46100 Burjassot,
         Val\`encia, Spain}
\email{gabriel@uv.es}

\thanks{The first author gratefully acknowledges support by the Deutsche
 Forschungsgemeinschaft -- Project-ID
 286237555-- TRR 195. The research of the second author is supported by
 Ministerio de Ciencia e Innovaci\'on PID2019-103854GB-I00 and FEDER funds.
 He thanks Attila Mar\'oti and Noelia Rizo for useful conversations on some
 aspects of the paper.}



\date{\today}

\begin{abstract}
We prove a strengthening of Brauer's height zero conjecture for principal
2-blocks with Galois automorphisms. This requires a new extension of the
It\^o--Michler theorem for the prime~2, again with Galois automorphisms.
We close, this time for odd primes $p$, with a new characterisation of
$p$-closed groups via the decomposition numbers of certain characters.
\end{abstract}

\maketitle


\section{Introduction}
Is there a strengthening of Brauer's Height Zero Conjecture with Galois
automorphisms? If $G$ is a finite group, $p$ is a prime and $B$ a $p$-block
of $G$ with defect group~$D$, Brauer's conjecture from 1955 asserts that $D$ is
abelian if and only if all complex irreducible characters in $B$ have height
zero. (Recall that $\chi\in\irr B$ has height zero if $\chi(1)_p=|G|_p/|D|$,
and that $\irro B$ denotes the set of irreducible characters in $B$ which have
height zero.) The proof of this statement has recently been completed \cite{MNST}.

The fundamental local/counting conjectures in representation theory of finite
groups were successfully strengthened in \cite{Na04} using the \emph{Frobenius
elements} of $\Gal(\QQ^\ab/\QQ)$: those which
send roots of unity of order not divisible by $p$ to some fixed $p$-power.
For $H \le \Gal(\QQ^\ab/\QQ)$, let $\Irr_H(B)$ denote the set of $H$-fixed
irreducible characters of~$B$. Then, is there an $H\ne1$ such that
$\Irr_H(B) \sbs \irro B$ if and only if $D$ is abelian?
If we pause to think on this for a moment, in the trivial case where $G$ is a
$p$-group, and therefore $\irr B$ consists of all the irreducible characters of
$G$, it seems natural to assume that $H$ fixes the $p$-power roots of unity.
\medskip

Let $\sigma\in \Gal(\QQ^\ab/\QQ)$ be the automorphism that fixes $2$-power
roots of unity and complex-conjugates odd-order roots of unity. If $B$ is a
$2$-block, let us denote by $\Irr_\sigma(B)$ the irreducible characters of~$B$
that are fixed by $\sigma$. The following strengthening of Brauer's height zero
conjecture is the main result of this paper.
 
\begin{thmA}
 Let $B$ be the principal $2$-block of a finite group $G$, and let
 $P\in\Syl_2(G)$. Then all characters in $\Irr_\sigma(B)$ have odd degree if
 and only if $P$ is abelian.
\end{thmA}

Theorem A does not extend to arbitrary 2-blocks, even if these possess
$\sigma$-invariant height zero characters, as shown by the non-principal
2-block of ${\tt SmallGroup}(96,13)$ in \cite{GAP}. Also, we cannot replace
$\sigma$ by complex conjugation (as shown by many 2-groups, for instance).
It might have some interest to remark that our $\sigma$ needs not be a
\emph{Frobenius} automorphism because it is not necessarily true that if
$m=|G|_{2’}$, then there exists an integer $n$ such that $2^n\equiv -1\pmod m$.
(For instance, if  $m=7$ or~15.)
\medskip

In order to obtain Theorem A, we shall need to prove the following extension of
the It\^o--Michler theorem for $p=2$. Here, for $G$ a finite group, we let
$\irrs G$ denote the set of $\sigma$-invariant irreducible characters of $G$.

\begin{thmB}
 Let $G$ be a finite group, and let $P \in \Syl_2(G)$. Then all characters
 in $\irrs G$ have odd degree if and only if $P$ is normal in $G$ and is
 abelian.
\end{thmB}

Theorem B resembles the main result of \cite{DNT}, where it is proved that if
all the real-valued irreducible characters of $G$, $\Irr_\RR(G)$, have odd
degree, then $G$ has a normal Sylow 2-subgroup. Let us point out now that the
sets $\Irr_\RR(G)$ and $\irrs G$ are in general different:
$\irrs G$ can be properly contained in $\Irr_\RR(G)$ (for instance,
if $G=D_{24}$), or the other way around (if $G=\GL_2(3)$).
\medskip

Unfortunately, we have not been able to find versions of Theorem~A or B for odd
primes and general finite groups, if they exist. (For solvable groups and only
for the normality of the Sylow subgroups, see \cite{Gr}.)
\medskip

In the final part of the paper, now for odd primes $p$, we derive a new
characterisation for a finite group to be $p$-closed. In \cite{MN}, we isolated
the normality condition in the It\^o--Michler theorem by proving that a Sylow
$p$-subgroup $P$ of $G$ is normal if and only if the irreducible constituents
of the permutation character $(1_P)^G$ have degree not divisible by~$p$.
Now we use the $p$-rational characters that lift modular characters of $G$
to prove the following result. Recall that an irreducible complex character
$\chi$ of $G$ is \emph{$p$-rational} if it has values in a cyclotomic field
$\QQ_m$, where $m$ is not divisible by $p$.

\begin{thmC}
 Let $G$ be a finite group, and let $p$ be an odd prime number.
 Let $P \in \Syl_p(G)$. Then $P\unlhd G$ if and only if every $p$-rational
 $\chi\in\Irr(G)$ with $\chi^0\in \IBr(G)$ has degree not divisible by~$p$.
\end{thmC}

Theorem C generalises the main result of \cite{NT12} (for odd primes). It does
not extend to $p=2$, though: for instance, the Mathieu groups $M_{22}$ and
$M_{24}$ do not have non-trivial characters lifting irreducible 2-modular
characters. Due to the result in \cite{MN} already mentioned,
it is tempting to explore possible relationships between the irreducible
constituents of $(1_P)^G$ and the $p$-rational characters that lift irreducible
Brauer characters. As shown by $\SL_2(3)$ for $p=3$, these sets do not seem
well related.
\medskip

The $p$-rational characters that lift irreducible Brauer characters constitute
an interesting object of study. I.~M.~Isaacs proved, for $p$-solvable groups
and $p$ odd, that every irreducible $p$-Brauer character has a unique
$p$-rational lift (\cite{Is74}). (So in $p$-solvable groups, these are abundant
characters.) Moreover, continuing in this case, the subnormal irreducible
constituents of these lifts are again $p$-rational lifts. Of course, Brauer
characters do not admit lifts in general, although it is not uncommon to find
$p$-rational characters that lift irreducible Brauer characters. In
Example~\ref{exmp:An} we exhibit an infinite family of cases showing that
$p$-rational lifts of modular characters are not necessarily unique.
On the other hand, for $p>3$, we have not been able to find a $p$-rational
character $\chi$ such that $\chi^0 \in \IBr(G)$ with a subnormal
constituent $\theta \in \Irr(N)$ such that $\theta^0$ is not irreducible.
\medskip

We have said above that Theorem A does not extend to arbitrary blocks. It is
not unreasonable, though, to think that in any 2-block with a
$\sigma$-invariant canonical character (these  are the defect zero character
of any root of the block, see \cite{Na98} for a definition), perhaps one can
strengthen Brauer's Height Zero Conjecture only using $\sigma$-invariant
characters. We have not attempted this here.
\medskip

We prove Theorem~B in Section~\ref{sec:thm B}, then Theorem~A in
Section~\ref{sec:thm A} using Theorem~B, and finally Theorem~C in
Section~\ref{sec:thm C}. In Section~\ref{sec:simple} we derive some facts about
characters of almost simple groups needed along the way.
\medskip

\noindent{\bf Acknowledgement:} We thank the referee for their pertinent
comments and questions which helped to considerably improve parts of this
paper.

\section{Characters of almost simple groups}   \label{sec:simple}
 
In this section we derive some results on characters of finite almost simple
groups with respect to the Galois automorphism $\sigma\in \Gal(\QQ^\ab/\QQ)$
that fixes $2$-power roots of unity and complex-conjugates odd-order roots of
unity. These will be needed in the proofs of our first two main theorems.
Throughout, $B_0(G)$ denotes the principal 2-block of a finite group $G$.

\subsection{Existence of $\sigma$-invariant characters}

We first discuss simple groups with abelian Sylow 2-subgroups; for this we
recall the following well-known facts:

\begin{lem}\label{3.1}
 Let $S$ be non-abelian simple with abelian Sylow $2$-subgroups.
 \begin{enumerate}
  \item[\rm(a)] If $G$ is quasi-simple with $G/\bZ(G)=S$ and $|\bZ(G)|=2$, then
   $G$ has non-abelian Sylow $2$-subgroups.
  \item[\rm(b)] If $S < G\le\Aut(S)$ with $G/S$ a $2$-group, then $G$ has
   non-abelian Sylow $2$-subgroups.
 \end{enumerate}
\end{lem}

\begin{proof}
By the classification theorem of Walter \cite{Wa69}, $S$ is one of $\PSL_2(q)$
with
$q\equiv\pm3\pmod8$, $\PSL_2(2^f)$ ($f\ge3$), $^2G_2(3^{2f+1})$ ($f\ge1$) or
$J_1$. Now the latter three do not have proper Schur covers, so in~(a) there is
nothing to prove for these. For $S=\PSL_2(q)$ with $q\equiv\pm3\pmod8$, the
only relevant central extension is $G=\SL_2(q)$, and this has quaternion
Sylow 2-subgroups of order~8. If $S$ has an even order outer automorphism,
then either $S=\PSL_2(q)$ with $q\equiv\pm3\pmod8$ and $G=\PGL_2(q)$ with
dihedral Sylow 2-subgroups of order~8, or $S=\PSL_2(2^f)$ with $f$ even, and
here the field automorphisms act as Galois automorphisms on a Sylow 2-subgroup
$P\cong \FF_{2^f}^+$ of $S$.
\end{proof}

\begin{prop}   \label{gunter 3}
 Let $S$ be non-abelian simple with abelian Sylow $2$-subgroups.
 \begin{enumerate}
  \item[\rm(a)] There exist $\chi_1\in\irrs S$ of even degree and
   $1_S\ne\chi_2\in\irrs{B_0(S)}$.
  \item[\rm(b)] If $G$ is quasi-simple with $G/\bZ(G)=S$ and $|\bZ(G)|=2$, then
   there is a faithful $\chi\in\irrs{B_0(G)}$ (hence of even degree).
  \item[\rm(c)] If $S < G\le\Aut(S)$ with $G/S$ a $2$-group, then there exists
   $\chi\in\irrs{B_0(G)}$ of even degree not lying over $1_S$.
 \end{enumerate}
\end{prop}

\begin{proof}
In (a), for $S=\PSL_2(q)$ with $3<q\equiv\pm3\pmod8$ we take for $\chi_1$ a
(real) character of degree $q-1$ labelled by an element of odd order in a torus
of order $q+1$ in the dual group $\PGL_2(q)$, for $\chi_2$ the Steinberg
character; for $^2G_2(q^2)$ with $q^2\ge27$ for $\chi_1$ a (real) character of
degree $(q^4-1)(q^2-\sqrt{3}q+1)$ (see \cite{Chv}) and for $\chi_2$ the Steinberg character;
for $J_1$ for $\chi_1$ one of the two (real) characters of degree~56, and
for $\chi_2$ the rational character of degree~209; and for $\PSL_2(q)$,
$q=2^f\ge8$, for $\chi_1$ the Steinberg character and for $\chi_2$ a real
character of degree $q-1$ labelled by an element of order $q+1$ in $\PGL_2(q)$.
Then in all cases, $\chi_1$ is $\sigma$-invariant of even degree, and
$\chi_2\in\Irr_\sigma(B_0(S))$.
\par
For (b), the only case is $G=\SL_2(q)$ with $q\equiv\pm3\pmod8$. Here the
claim follows from the known character table of $G$ (see
\cite[Tab.~2.6]{GM20}). Alternatively, the dual
group $\PGL_2(q)$ contains an element $s$ of order~4 not lying in $\PSL_2(q)$.
Since $s$ is regular semisimple, by Lusztig's Jordan decomposition
\cite[2.6.22]{GM20} it labels a faithful irreducible Deligne--Lusztig character
$\chi$ of $\SL_2(q)$. Since $s$ is a rational element, $\chi$ is
rational (by \cite[Cor.~3.3.14]{GM20}) and thus $\sigma$-invariant. Also, as
$s$ is a 2-element, $\chi\in\Irr(B_0(G))$ by \cite[Thm.~21.14]{CE}.   \par
In (c), for $S=\PSL_2(2^f)$ all outer automorphisms are field automorphisms.
Here, the irreducible induction of a character $\chi\in\Irr(S)$ corresponding
to an element of order $q+1$ in $\PGL_2(q)$ is as desired. For $\PSL_2(q)$,
$3<q\equiv\pm3\pmod8$, we have $G=\PGL_2(q)$. Here we take for $\chi$ an
irreducible character of degree $q-1$ labelled by an element of order~4 in
the dual group $\SL_2(q)$. Again, $\chi$ is rational by \cite[Cor.~3.3.14]{GM20}
and lies in $B_0(G)$ by \cite[Thm.~21.14]{CE}.
\end{proof}

The main result of this section is the following; its proof will occupy the
remainder of this section:

\begin{thm}   \label{thm:gunter 2+4}
 Let $S$ be simple with non-abelian Sylow $2$-subgroups and $S\le G\le\Aut(S)$
 with $G/S$ a $2$-group. Then there exists $\chi\in\Irr_\sigma(B_0(G))$ of even
 degree not lying over $1_S$.
\end{thm}

In this paper, if $N$ is a normal subgroup of $G$ and $\theta \in \irr N$,
we will denote by $G_\theta$ or $I_G(\theta)$ the stabilizer of $\theta$
in~$G$.  We will make use of the following criterion:

\begin{lem}   \label{lem:odd}
 In the situation of Theorem~{\rm\ref{thm:gunter 2+4}}, let
 $1_S\ne\theta\in\Irr(S)$ of odd degree be such that $\theta^\sigma=\theta^g$
 for some $g \in G$. If $\theta$ is not $G$-invariant and lies in the principal
 2-block of $S$, then $\theta^G$ has a constituent $\chi$ satisfying the
 conclusion of Theorem~{\rm\ref{thm:gunter 2+4}}.
\end{lem}

\begin{proof}
Since $S$ is perfect, the determinantal order $o(\theta)$ is~1, so $|G:S|$ is
prime to $o(\theta)\theta(1)$. Let $G_\theta$ be the stabiliser of $\theta$
in $G$. By \cite[Cor.~6.28]{Isaacs} there exists a unique
$\eta\in\Irr(G_\theta)$ extending $\theta$ with $o(\eta)=1$.
Since $\theta^\sigma=\theta^g$ for some $g\in G$ we have
$\theta^{\sigma g^{-1}}=\theta$, so $\eta^{\sigma g^{-1}}$ is an extension of
$\theta$ with determinantal order 1. Thus $\eta^{\sigma g^{-1}}=\eta$, and
therefore $\eta^\sigma=\eta^g$. Now $\chi:=\eta^G\in\Irr(G)$ has even degree
because $G_\theta<G$, and
$\chi^\sigma=(\eta^\sigma)^G=(\eta^g)^G=\chi^g=\chi$. Finally $\chi$ lies in
$B_0(G)$ since this is the only 2-block covering $B_0(S)$ by \cite[Cor.~9.6]{Na98}.
\end{proof}

\subsection{Alternating, sporadic, Ree and Suzuki groups}
Let's first get some easy cases out of the way.

\begin{prop}   \label{prop:alt}
 Theorem~{\rm\ref{thm:gunter 2+4}} holds for $S$ either an alternating group
 $\fA_n$ with $n\ge5$ or a sporadic simple group.
\end{prop}

\begin{proof}
For the sporadic groups and their automorphism groups, the known character
tables \cite{GAP} allow one to check the claim directly. Let now $G=\fS_n$ with
$n\ge5$. Let $\la$ be the partition $(n-2,2)$ when $n\equiv0,3\pmod4$ and
$(n-2,1^2)$ when $n\equiv1,2\pmod4$. Then by
the hook formula, the irreducible character $\chi\in\Irr(G)$ labelled by
$\lambda$ has even degree. Since $\la$ has 2-core of size~1 or~0, $\chi$
lies in the principal 2-block of~$G$. Since all characters of $\fS_n$ are
rational valued, this proves our claim in this case.   \par
The group $\fA_5$ has abelian Sylow 2-subgroups. For $n\ge6$, $\la$ is not
self-conjugate and so $\chi$ restricts irreducibly to $S=\fA_n$ and hence also
proves our claim there. The only remaining case is $S=\fA_6$ and $G$ one of
$\PGL_2(9)$, $M_{10}$ or $\Aut(\fS_6)$. For these groups, the extensions of
$\chi_S$ to $G$ have values in $\QQ(\zeta_8)$ and thus are $\sigma$-invariant.
\end{proof}

\begin{prop}   \label{prop:Ree}
 Theorem~{\rm\ref{thm:gunter 2+4}} holds for $S$ a Suzuki or Ree group.
\end{prop}

\begin{proof}
The groups $^2G_2(3^{2f+1})$ have abelian Sylow 2-subgroups and thus need not
be considered. For $S=\tw2F_4(2)'$ the assertion is easily checked. Otherwise,
$\Out(S)$ has odd order whence $G=S$. Here, all characters apart from the
Steinberg character lie in the principal 2-block by \cite[Thm.~6.18]{CE}.
For $\tw2B_2(2^{2f+1})$ the two cuspidal unipotent characters have values in
$\QQ(\sqrt{-1})$ and thus work; for $\tw2F_4(2^{2f+1})$, we may take any of the
three unipotent characters of even degree in the principal series.
\end{proof}

\subsection{Groups of Lie type in characteristic~2}
In this section we work in the following setting: $\bG$ is a simple algebraic
group of adjoint type over an algebraically closed field of characteristic~$2$
with a Frobenius endomorphism $F$ such that $S=[\bG^F,\bG^F]$ is non-abelian
simple. We denote by $\bG^*$ a group in duality with $\bG$, with corresponding
Frobenius map also denoted~$F$. Thus, $\bG^*$ is of simply connected type. The
characters whose existence is stipulated in Theorem~\ref{thm:gunter 2+4} will
mainly be constructed as follows:

\begin{lem}   \label{lem:elt}
 Let $\bG,\bG^*$ and $S$ be as above. Let $1\ne s\in\bG^{*F}$ be a semisimple
 element with connected centraliser in $\bG^*/\bZ(\bG^*)$ and let
 $\chi\in\Irr(\bG^F)$ be the associated regular character. Then
 \begin{enumerate}
  \item[\rm(a)] $\chi$ restricts irreducibly to $S$;
  \item[\rm(b)] $\chi_S\in\Irr(B_0(S))$;  
  \item[\rm(c)] if $s$ is real then $\chi$ is $\sigma$-stable;
  \item[\rm(d)] if $s$ is not regular then $\chi(1)$ is even; and
  \item[\rm(e)] for any automorphism $\gamma$ of $\bG^F$ induced by a Frobenius
   map $F_0:\bG\to\bG$ commuting with $F$, $\chi$ has a $\sigma$-stable
   extension to $\bG^F\langle\gamma\rangle_\chi$.
 \end{enumerate}
\end{lem}

\begin{proof}
As $\bG$ has trivial centre, there is a unique (irreducible) regular character
$\chi$ of $\bG^F$ attached to $s$, see \cite[12.4.10]{DM20}. Let $\bH\surj\bG$
be a simply connected covering of $\bG$, and $\bH\hookrightarrow\tilde\bH$ a
regular embedding and choose corresponding Frobenius maps on $\bH$ and
$\tilde\bH$, also denoted~$F$. The natural maps
$\bH\hookrightarrow\tilde\bH\surj\tilde\bH/\bZ(\tilde\bH)
 \cong\bH/\bZ(\bH)\cong\bG$ 
induce $F$-equivariant maps $\bG^*\hookrightarrow\tilde\bH^*\surj\bH^*$ of dual
groups, where $\bH^*\cong\bG^*/\bZ(\bG^*)$. By assumption, the image
$\bar s\in\bH^*$ of $s$ has connected centraliser. Let $\tilde\chi$ denote the
inflation of $\chi$ to $\tilde\bH^F$. By \cite[Prop.~2.6.17 and Lemma~2.6.20]{GM20} the
number of constituents of $\tilde\chi|_{\bH^F}$ equals the number of connected
components of $C_{\bH^*}(\bar s)$, that is, $\tilde\chi$ restricts irreducibly
to $\bH^F$. Thus $\chi$ restricts irreducibly to
$S=[\bG^F,\bG^F]\cong \bH^F/\bZ(\bH)^F$ whence we obtain~(a).   \par
Now, in defining characteristic all irreducible characters of $\bG^F$ apart
from those above the Steinberg character of $S$ lie in the principal 2-block
\cite[Thm.~6.18]{CE}. Since the Steinberg character is unipotent while
$\chi\in\cE(G,s)$ with $s\ne1$, we have $\chi\in\Irr(B_0(\bG^F))$
and hence also $\chi_S\in\Irr(B_0(S))$. For~(c) observe that $\chi$ is uniform
by its very definition as a linear combination of Deligne--Lusztig characters
\cite[Thm.~3.4.16]{GM20}. Then by the character formula for Deligne--Lusztig
characters \cite[Thm.~2.2.16]{GM20} the values of $\chi$ lie in the field
generated by values of linear characters of order dividing $o(s)$ of various
tori of $\bG^F$. Thus, since $s$ is semisimple and hence of odd order, the
character values of $\chi$ lie in an extension generated by odd-order roots of
unity. Moreover, if $s$ is real then $\chi$ is real valued using
\cite[Prop.~3.3.15]{GM20}, whence it is $\sigma$-stable. By the degree formula
\cite[Thm.~3.4.16(c) and Cor.~2.6.6]{GM20} the degree $\chi(1)$ is divisible by
$|C_{\bG^*}(s)^F|_2$, hence even if $s$ is not regular.   \par
For (e) note that we are in the setting of \cite[\S6.5]{Ru21}: Since $\chi$ is
regular, its Alvis--Curtis dual $\chi^0$ is semisimple and hence lies in the
image of the map $f$ from \cite[Thm~5.7]{Ru21}. Now a maximal unipotent
subgroup of $\bG^F$ is a (Sylow) 2-group, hence all of its characters are
$\sigma$-invariant. This shows that our Galois automorphism $\sigma$ satisfies
the Assumption~6.6 of \cite{Ru21} with $\tilde t=1$. In particular $\mu_\sigma$
in \cite[Rem.~6.7]{Ru21} is then the trivial character. Our assertion~(e) is
now \cite[Prop.~6.8]{Ru21}. For this note that its proof applies to $\sigma$
since it only uses \cite[Prop.~6.2]{Ru21} which is formulated for arbitrary
Galois automorphisms. Moreover, the proof actually first shows that the
Alvis--Curtis dual $\chi$ of $\chi_0$ has the desired property.
\end{proof}

\begin{prop}   \label{prop:char 2}
 Theorem~{\rm\ref{thm:gunter 2+4}} holds for $S$ a simple group of Lie type in
 characteristic~$2$.
\end{prop}

\begin{proof}
Let $S$ be as in the statement. By Proposition~\ref{prop:Ree} we may assume
$S$ is not a Suzuki or Ree group. Thus there is a simple algebraic group
$\bG$ of adjoint type in characteristic $2$ with a Frobenius
endomorphism~$F:\bG\to\bG$ such that $S=[\bG^F,\bG^F]$. We have $\bG$ is not of
type $A_1$ since Sylow 2-subgroups of $S$ are non-abelian.  \par
For $S=\PSL_3(2)$ or $\PSL_4(2)$ the claim is easily verified. For all other
cases we claim $\bG^{*F}$ contains a non-trivial real non-regular semisimple
element with connected centraliser in $\bG^*/\bZ(\bG^*)$. Indeed, in groups of
type $A$ we choose an element with non-trivial eigenvalues the two primitive
third roots of unity, and all other eigenvalues~1, in groups of type $E_6$ we
take an element of order five in a subsystem subgroup of type $A_3$, so
centralised at least by an $A_1$-subgroup, and in
all other types we take any element of order~3 in an $\SL_2$-subgroup (this has
connected centraliser by application of \cite[Prop.~14.20]{MT}).
Now a Sylow 2-subgroup of the outer automorphism group of~$S$ is generated by
graph automorphisms together with Frobenius endomorphisms of~$\bG$ commuting
with~$F$ (see \cite[Thm.~2.5.1]{GLS}). Thus, by Lemma~\ref{lem:elt} we are done
unless $G$ induces non-trivial graph automorphisms, that is, $S$ is of type
$A_{n-1}$ ($n\ge3$), $D_n$ ($n\ge4$) or $E_6$. So now assume we are in the
latter cases.

For $S=\PSL_n(q)$ let $t\in\PGL_n(q)$ be the image of an element
$\tilde t\in\GL_n(q)$ of order $q^n-1$ (if $n$ is odd), respectively
$q^{n-1}-1$ (if $n$ is even), and let $s:=t^d\in[\PGL_n(q),\PGL_n(q)]$ with
$d=\gcd(n,q-1)$. Then $s$ has connected centraliser in $\PGL_n\cong \bG^*/\bZ(\bG^*)$. It
parametrises a semisimple character $\chi$ of $S$ (see \cite[2.6.10]{GM20})
whose values lie in a cyclotomic field generated by odd-order roots of unity.
The eigenvalues of $\tilde t$ form an orbit under $\Gal(\FF_{q^n}/\FF_q)$ of a
generator of $\FF_{q^n}^\times$ if $n$ is odd, respectively an orbit
under $\Gal(\FF_{q^{n-1}}/\FF_q)$ of a generator of $\FF_{q^{n-1}}^\times$
together with~1, if $n$ is even.
From this it follows that the class of $s$ is not invariant under non-trivial
2-power order field automorphisms, nor under their product with the
transpose-inverse automorphism, sending $s$ to its inverse (up to
conjugation). Thus, using \cite[Prop.~7.2]{Tay18}, $\chi$ is not invariant
under the dual automorphisms of $S$, so $\chi^G\in\Irr(B_0(G))$ is of even
degree and real, whence $\sigma$-stable.
\par
Let $S$ be of type $D_n(q)$. Since we are in characteristic~2 we may identify
$\bG^{*F}$ with $\SO_{2n}^+(q)$. Let $s$ be an element of order $q^n-1$ in the
stabiliser $\GL_n(q)$ of a maximal totally isotropic subspace of $\bG^{*F}$.
Being semisimple, $s$ has connected centraliser in
$\bG^*/\bZ(\bG^*)$. The eigenvalues of $s$ in this representation are
then a Galois orbit of a generator of $\FF_{q^n}^\times$ together with their
inverses. It follows that the set of eigenvalues is not invariant under any
non-trivial automorphism of $\FF_q$, hence the class of $s$ is not invariant
under any non-trivial field automorphism of $\bG^{*F}$. Also, if $n$ is even,
then $s$ is real but the
graph automorphism interchanges the two conjugacy classes of stabilisers
of totally singular subspaces, hence does not fix the class of~$s$. If $n$ is
odd, $s$ is non-real but the graph automorphism of $S$ induces the
transpose-inverse automorphism of $\GL_n(q)$, so conjugates $s$ to its inverse.
So in either case the semisimple character $\chi$ of $S$ labelled by~$s$ has 
$\chi^G\in\Irr_\sigma(B_0(G))$ of even degree.
\par
Finally, for $S=E_6(q)$ let $s$ be the third power of a generator of a maximal
torus $T$ of $\bG^{*F}$ of order $q^6+q^3+1$, contained in a subsystem subgroup
$H=A_2(q^3)$. Since the order of $s$ is prime to~3, its image has connected
centraliser in $\bG^*/\bZ(\bG^*)$ by \cite[Prop.~14.20]{MT}. By definition the
order of $s$ is divisible by a Zsigmondy prime for $q^9-1$, so $s$ lies in a
unique maximal torus of $\bG^{*F}$ and so is regular semisimple.
Thus $N_S(\langle s\rangle)=N_S(T)$. Since $|N_S(T)/T|=9$, $s$ is non-real. A
computation in the root system of $\bG^*$ using \Chevie\ \cite{Chv} shows that
the graph automorphism of order~2 of $\bG^*$ induces on $H=A_2(q^3)$ the
transpose-inverse automorphism
and thus inverts~$s$. Computation in $A_2$ shows that the class of $s$ is also
not invariant under non-trivial field automorphisms of 2-power order or
products of these with the graph automorphism, so we conclude as before.
\end{proof}

\subsection{Groups of Lie type in odd characteristic}

We now prove Theorem~\ref{thm:gunter 2+4} for simple groups of Lie type in odd
characteristic. Let $\bG$ be a simple algebraic group of adjoint type over a
field of odd characteristic, with a Frobenius endomorphism $F$ with respect to
an $\FF_q$-structure, and let $(\bG^*,F)$ be dual to $(\bG,F)$. We set
$S:=[\bG^F,\bG^F]$. Similar to the approach in the even characteristic case,
the following will be our main source of suitable characters:

\begin{lem}   \label{lem:2-elt}
 Let $\bG,\bG^*,S$ be as above. Let $1\ne s\in\bG^{*F}$ be a $2$-element with
 connected centraliser in $\bG^*/\bZ(\bG^*)$. Then the associated semisimple
 character $\chi\in\Irr(\bG^F)$ satisfies:
 \begin{enumerate}
  \item[\rm(a)] $\chi$ restricts irreducibly to $S$;
  \item[\rm(b)] $\chi_S\in\Irr(B_0(S))$;  
  \item[\rm(c)] $\chi$ is $\sigma$-stable;
  \item[\rm(d)] $\chi(1)$ is odd if and only if $s$ lies in the centre of a
   Sylow $2$-subgroup of $\bG^{*F}$; and
  \item[\rm(e)] for any automorphism $\gamma$ of $\bG^F$ induced by a Frobenius
   map $F_0:\bG\to\bG$ commuting with $F$, $\chi$ has a $\sigma$-stable
   extension to $\bG^F\langle\gamma\rangle_\chi$.
 \end{enumerate}
\end{lem}

\begin{proof}
Since $\bG$ is of adjoint type we have $\bZ(\bG)=1$ and thus there is a unique
semisimple character $\chi\in\Irr(\bG^F)$ in the Lusztig series $\cE(\bG^F,s)$
(see \cite[2.6.10]{GM20}). Since the centraliser of $s$ in $\bG^*/\bZ(\bG^*)$ is
connected, the argument in the proof of the corresponding statement in
Lemma~\ref{lem:elt} gives~(a). For $s$ a 2-element, $\chi$ lies in the
principal 2-block by \cite[Thm.~B]{En00}. Since $\bG^*$ is of simply connected
type, $C_{\bG^*}(s)$ is connected, so by
the definition of $\chi$ as a linear combination of Deligne--Lusztig characters
\cite[Def.~12.4.2]{DM20}, the values of $\chi$ only involve 2-power roots of
unity, and thus $\chi$ is $\sigma$-stable, whence~(c). Part~(d) is a direct
consequence of the degree formula \cite[Cor.~2.6.6]{GM20} for Lusztig's Jordan
decomposition.   \par
For~(e) we again use \cite{Ru21}. Observe that $\bG$ is of adjoint type. In the
setting of \cite[3.1]{Ru21} choose $\phi_0\in\Irr((\FF_{q^N},+))$ by
decomposing $(\FF_{q^N},+)=(\FF_p,+)\oplus C$ such that $-C=C$ and defining
$$\phi_0(a):=\begin{cases} \exp(2\pi\bi\, a/p)& \text{if $a\in\FF_p$,}\\
                           1& \text{if $a\in C$.}\end{cases}$$
As $\sigma$ complex conjugates odd order roots of unity,
$\tw\sigma\phi_0(a)=\overline{\phi_0(a)} =\phi_0(-a)$ for all $a\in \FF_{q^N}$
and so in the notation of \cite[Lemma 3.2]{Ru21},
$$\tw\sigma\phi_i(x_i(a))=\tw\sigma\phi_0(c_ia)=\phi_0(-c_ia)=\phi_i(x_i(-a))
  \qquad \text{for all $a\in \FF_{q^N}$}.$$
Let $F_0:\bG\to\bG$ denote a Frobenius map with respect to an $\FF_p$-structure
and such that $F=F_0^f\rho$ for some $f\ge1$ and some graph automorphism
$\rho$ of $\bG$. Let $\bT\le\bG$ be a maximally $F_0$-split torus. 
Since $\bG$ is of adjoint type there is $\tilde t\in\bT^{F_0}$ such that
$\alpha_i(t)=-1$ for all simple roots $\alpha_i$ and thus
$x_i(a)^{\tilde t}=x_i(-a)$ for all $a\in \FF_{q^N}$.
This means that $\tw\sigma\phi_i(x_i(a))=\phi_i(x_i(a)^{\tilde t})$ for all $i$
and so $\tw\sigma\phi=\phi^{\tilde t}$ for all $\phi\in\Irr(\bU/[\bU,\bU]^F)$.

Now note that $\tilde t$ is invariant under all graph automorphisms of $\bG$
and so $d(\tilde t)\tilde t^{-1}=1$ for all graph and field automorphisms
$d$ of $\bG^F$. Thus Assumption~6.6 in \cite{Ru21} is satisfied. Then
\cite[Prop.~6.8]{Ru21} shows that the semisimple character $\chi$ has a
$\sigma$-stable extension as claimed.
\end{proof}

\begin{prop}   \label{prop:An}
 Theorem~{\rm\ref{thm:gunter 2+4}} holds for $S=\PSL_n(q)$ and $\PSU_n(q)$ with
 $n\ge2$ and $q$ odd.
\end{prop}

\begin{proof}
When $n=2$, since Sylow 2-subgroups of $S$ are assumed non-abelian we have
$q\equiv\pm1\pmod8$. Here from the character table in \cite[Tab.~2.6]{GM20}
one sees that any element $s\in\bG^{*F}\cong\SL_2(q)$ of order 8 will satisfy
the assumptions of Lemma~\ref{lem:2-elt}.

So now assume $n\ge3$. Let $\zeta\in\FF_{q^2}^\times$ be a generator of the
Sylow 2-subgroup. Let
$s\in\SL_n(q)$ be an element with eigenvalues $\zeta,\zeta^q,\zeta^{-q-1}$
and all other eigenvalues equal to~1. Then $s$ is a 2-element that
is conjugate to its $q$th power, but not to any other of its primitive powers
modulo scalars, unless $n=4$ and $q\equiv3\pmod4$. Let us for the moment
exclude that latter case. Then the image of the conjugacy class of $s$ in
$\PSL_n(q)$ is not invariant under non-trivial field automorphisms of
$\PSL_n(q)$. The transpose-inverse graph automorphism acts by inverting the
eigenvalues of semisimple elements, and thus it ensues that the class of $s$
is neither invariant under products of field automorphisms with this graph
automorphism. Furthermore, the
image of $s$ has connected centraliser in $\bG^*/\bZ(\bG^*)\cong\PGL_n$. Let
$\chi\in\Irr(\PGL_n(q))$ be the semisimple character associated to~$s$. Then
$\chi\in\Irr_\sigma(B_0(\PGL_n(q)))$ by Lemma~\ref{lem:2-elt} and $\chi$
restricts irreducibly to $S=\PSL_n(q)$. By what we said above about $s$,
$\psi:=\chi_S\in\Irr_\sigma(B_0(S))$ is not invariant under 2-power order field
or graph automorphisms of $S$. Thus, the full
stabiliser of $\psi$ in $\Out(S)$ is the group of diagonal automorphisms. So
for any $G$ as in Theorem~\ref{thm:gunter 2+4}, $\psi$ extends to
$\tilde\chi:=\chi_{G\cap\PGL_n(q)}\in\Irr(G\cap\PGL_n(q))$ and from there
induces irreducibly to $G$, whence $\tilde\chi^G$ is as desired.
\par
Now assume that $n=4$ and $q\equiv3\pmod4$. Since $q$ is not a square,
$S=\PSL_4(q)$ does not have even order field automorphisms. Let $\chi$ be the
unipotent character of $\PGL_4(q)$ labelled by the partition~$(2^2)$. Then
$\chi$ has even degree, lies in the principal 2-block and restricts irreducibly
to~$S$. Furthermore, $\chi$ extends to a rational character $\tchi$ of the
extension of $\PGL_4(q)$ by the graph automorphism, by
\cite[Thm.~II.3.3]{DM85} combined with \cite[Thm~3.8]{SF19}. Then $\tchi_G$ is as desired.
\par
Similarly, we let $s\in\SU_n(q)$ be an element with non-trivial eigenvalues
$\zeta,\zeta^{-q},\zeta^{q-1}$; then precisely the same
argument as before applies to show our claim for $S=\PSU_n(q)$.
\end{proof}

\begin{prop}   \label{prop:exc}
 Theorem~{\rm\ref{thm:gunter 2+4}} holds for $S$ of exceptional
 Lie type in odd characteristic, not of type $\tw{(2)}E_6$ or $E_7$.
\end{prop}

\begin{proof}
Let $S$ be simple of exceptional type. By Proposition~\ref{prop:Ree} we may
assume that $S=[\bG^F,\bG^F]$ for $\bG, F$ as above. Let $s\in\bG^{*F}$ be
a 2-element that is not 2-central. Since $\bZ(\bG^*)=1$ the centraliser of
$s$ in the simply connected group $\bG^*/\bZ(\bG^*)=\bG^*$ is connected. If
moreover $G$ only contains automorphisms induced by Frobenius maps, then by
Lemma~\ref{lem:2-elt}, the semisimple character $\chi\in\cE(\bG^F,s)$ does the
job. So we are left with the groups of type $G_2$ in characteristic~3.
\par
For $S=G_2(3^n)$ we need to discuss extensions $G$ involving the exceptional
graph automorphism $\sigma$ of order~2, or its product with a field
automorphism. Let $H\le \bG^{*F}\cong\bG^F=S$ be a subsystem subgroup of type
$A_1^2$. Let $s\in H$ be an element of maximal 2-power order in the short root
$A_1$-factor. Then $s$ centralises a long root element, thus the image of $s$
under $\sigma$ centralises a short root element, whence neither $\sigma$ nor
its product with a field automorphism  can fix the class of $s$. Moreover,
since $s$ is of maximal 2-power order in $A_1(q)$, its class is not fixed by
any non-trivial 2-power order field automorphism of $S$. Thus, the semisimple
character of $S$ labelled by $s$ induces irreducibly to $G$, lies in the
principal block, and is real as any semisimple element of an $A_1$-type group
is conjugate to its inverse.
\end{proof}

\begin{prop}   \label{prop:non-square}
 Theorem~{\rm\ref{thm:gunter 2+4}} holds for $S$ of Lie type in odd
 characteristic if $q$ is not a square.
\end{prop}

\begin{proof}
If $q$ is not a square then $S$ does not possess even order field automorphisms.
So in this case only products of diagonal and graph automorphisms can be
present in $G$.
Then we let $\chi$ be a rational unipotent character of $\bG^F$ of even degree
in the principal $2$-block of $\bG^F$. More specifically, for $\bG$ of type
$E_6$ we take $\chi$ to be $\phi_{20,2}$ for the untwisted and $\phi_{4,1}$
for the twisted version, and for $\bG$ of type $E_7$ we take $\phi_{210,6}$. In
view of Proposition~\ref{prop:exc} these are the only groups of exceptional
type we need to discuss. For $\bG$ of classical type we take $\chi$ labelled by
the symbol $\binom{0,1,n}{-}$ for types $B_n$ or $C_n$ ($n\ge2$), by
$\binom{0,1,2,n-1}{-}$ for type $D_n$ ($n\ge4$), by $\binom{1,n-1}{-}$ for
type $\tw2D_n$ with $n\ge4$ even, and by $\binom{0,1,n}{1}$ when $n\ge5$ is
odd. These are of even degree by \cite[Prop.~4.4.15]{GM20}.

All of the above lie in the principal block by \cite{En00}. The restriction of
the unipotent character $\chi$ to $S=[\bG^F,\bG^F]$ is still irreducible, hence
of even degree, and lies in the principal block of $S$. Finally, by
\cite[Cor.~II.3.4]{DM85} combined with \cite[Thm~3.8]{SF19}, $\chi$ extends to
a rational character of the extension of $\bG^F$ with a graph automorphism of
order~2, if such exists. Thus $\chi|_S$ is as required.
\end{proof}

\begin{prop}   \label{prop:square}
 Theorem~{\rm\ref{thm:gunter 2+4}} holds for $S$ of Lie type in odd
 characteristic if $q$ is a square.
\end{prop}

\begin{proof}
Here we have $q\equiv1\pmod8$. For $S$ of type $E_7$, $B_n$ ($n\ge2$) or
$C_n$ ($n\ge3$), let $\bL$ be a split Levi subgroup of~$\bG^*$ of type $E_6$,
$C_{n-1}$, $B_{n-1}$ respectively. Let $s\in \bZ(\bL)^F$ of order~8. As
$|\bZ(\bG^*)|=2$ the image of $s$ in $\bG^*/\bZ(\bG^*)$ has order at least~4 and
hence $C_{\bG^*}(s)$ is connected (otherwise, since $\bL$ is a maximal Levi
subgroup, $s$ would have to be quasi-isolated, but by \cite[Tab.~2]{Bo05}
there do not exist quasi-isolated elements of order at least~4 in $\bG^*$.
Inspection of the order formulae shows $s$ is not 2-central, so we
may conclude with Lemma~\ref{lem:2-elt}.   \par
Assume $\bG$ is of type $E_6$. If $G$ does not involve a graph automorphism
(of order~2) then we may argue exactly as in the proof of
Proposition~\ref{prop:exc}. If $G$ does induce a graph automorphism, and hence
$\bG^F=E_6(q)$, then let $s\in\bG^{*F}$ be an element of order~8 in the centre
of a split Levi subgroup $\bL$ of $\bG^*$ of type $D_5$. Since $\bL^F$
contains a Sylow 2-subgroup of $\bG^{*F}$ (by inspection of the order
formulae), the corresponding semisimple character $\chi$ of $S$ lies in
$\Irr_\sigma(B_0(S))$ and has odd degree by Lemma~\ref{lem:2-elt}. Also, since
$C_{\bG^*}(s)^F=L= N_{\bG^{*F}}(L)$, and $s$ is central in $L$, $s$ is not
conjugate to its inverse in $\bG^*$. But a calculation in the root system shows
that the graph automorphism of $\bG^*$ acts by inversion on $\bZ(\bL)$ and hence
inverts~$s$. Thus Lemma~\ref{lem:odd} applies.   \par
For $\bG$ of type $D_n$, $n\ge4$, let $\bL$ be a split Levi subgroup of $\bG^*$
of type $D_{n-1}$ and again $s\in \bZ(\bL)^F$ of order~8. If either
$n$ is even, or $\bG^F$ is of twisted type, then $\bZ(\bG^*)^F$ is elementary
abelian and so the image $\bar s$ of $s$ in $\bG^*/\bZ(\bG^*)$ has order at
least~4. According to \cite[Tab.~2]{Bo05} there are no quasi-isolated elements
of that order with centraliser containing $\bL$, so $\bar s$ has connected
centraliser. When $n$ is odd and $\bG^F$ is untwisted, let $\bL$ be a split
Levi subgroup of $\bG^*$ of type $A_{n-1}$ and again $s\in \bZ(\bL)^F$ of
order~8. In either case we can then conclude as before unless $G$ also induces
a graph automorphism of order~2.   \par
In the latter case, $\bG^F$ is untwisted. Let $\zeta\in\FF_{q^2}^\times$ be a
generator of the Sylow 2-subgroup (of order at least~16). Let
$t\in H:=\SO_{2n}^+(q)$ be an element with eigenvalues $\zeta,\zeta^q$ and
$n-2$ times $\zeta^2$, and their inverses, in the natural matrix representation.
Then $s:=t^2$ is not invariant under non-trivial 2-power order field
automorphisms. Furthermore $s$ has centraliser $\GL_1(q)^2\GL_{n-2}(q)$ in $H$
as well as in $\GO_{2n}^+(q)$, since $s$ has no eigenvalues $\pm1$, which means
that its class is not invariant under the graph automorphism of order~2 induced
by $\GO_{2n}^+(q)$. Since the graph automorphism preserves the eigenvalues of
semisimple elements, this then shows that the class of~$s$ is also not
invariant under products of field automorphisms with the graph automorphism.
Let $\chi\in\Irr_\sigma(B_0(\SO_{2n}^+(q)))$
be the corresponding semisimple character (noting that $H=\SO_{2n}^+(q)$ is
self-dual), see Lemma~\ref{lem:2-elt}. Since $s$ has connected centraliser,
$\chi$ restricts irreducibly to the derived subgroup of $\SO_{2n}^+(q)$.
Moreover, as $s$
is a square in $H$, it lies in $[H,H]$, so $\chi$ has $\bZ(\SO_{2n}^+(q))$ in
its kernel and can be considered as a character of $S=\OO_{2n}^+(q)$. Again
since $s$ has connected centraliser, $\chi$ also extends to a $\sigma$-stable
character of the adjoint type group $\bG^F$. By our earlier remarks on $\chi$,
all of these characters have trivial stabiliser in the group generated by graph
and field automorphisms.
So their induction to $G$ provides a character as claimed.
\end{proof}

By the classification of finite simple groups, the proof of
Theorem~\ref{thm:gunter 2+4} is now complete.

\section{$2$-closed groups}   \label{sec:thm B}
In this section, we improve on the It\^o--Michler Theorem for $p=2$ which
asserts that a finite group $G$ has a normal and abelian Sylow 2-subgroup
if and only if all the irreducible complex characters of $G$ have odd degree.
As we have mentioned,
our main result in this section and its proof resemble Theorem~A of \cite{DNT},
where it is shown that if all real-valued irreducible characters of $G$ have
odd degree then $G$ is 2-closed. Of course our Galois automorphism $\sigma$
behaves like complex conjugation but only on odd order roots of unity. This
difference allows us to fully generalise It\^o--Michler in both directions
(which cannot be done by using complex conjugation) but poses additional
complications. One of them arises from the fact that, unlike
complex conjugation, $\sigma$ does not act on $p$-Brauer characters for $p$
odd, as shown by the group $2.\fA_6.2_2$ in characteristic~3, for example.
 
Our notation for characters follows \cite{Isaacs} and \cite{Na18}, and our
notation for blocks and Brauer characters follows \cite{Na98}.
\medskip
 
We shall use below that if $\gamma\in\Gal(\QQ^\ab/\QQ)$ complex-conjugates
odd-order roots of unity, $n\ge 1$ is any integer, and $\QQ_n$ is the $n$-th
cyclotomic field, then the restriction $\tau=\gamma|_{\QQ_n}$ has order a power
of~2. This follows from the fact that $\tau^2$ fixes $2'$-roots of unity and
the Galois group $\Gal(\QQ_n/\QQ_{n_{2'}})$ is a 2-group. (In this paper, $n_2$
is the largest power of 2 dividing $n$, and $n_{2'}=n/n_2$.)
\medskip
 
Recall that $\Gal(\QQ^\ab/\QQ)$ permutes the $p$-blocks of any finite group
$G$ (but as we said, not the $p$-Brauer characters).
 
\begin{lem}\label{bl}
 Suppose that $\sigma \in\Gal(\QQ^\ab/\QQ)$ complex-conjugates odd-order roots
 of unity. Let $G$ be a finite group. Then the principal $2$-block of $G$ is the
 only $\sigma$-invariant $2$-block of $G$ with maximal defect.
\end{lem}

\begin{proof}
Suppose that $B$ is a $\sigma$-invariant 2-block of $G$ with defect group
$P\in \Syl_2(G)$. Let $b$ be its Brauer first main correspondent.
It easily follows that $(b^\sigma)^G=B^\sigma$, since Galois action commutes
with Brauer induction. By the uniqueness in Brauer's first main theorem,
and the third main theorem, we may assume that $P \unlhd G$. Now, let
$\theta\in\irr{\oh{2'}G}$ such that the 2-block $\{\theta\}$ is covered by $B$. 
Then $\{\theta^\sigma\}$ is covered by $B^\sigma=B$, and it follows that there
exists $g \in G$ such that $\theta^\sigma=\bar\theta=\theta^g$. Now, using that
the restriction of $\sigma$ to $\QQ_{|G|}$ has 2-power order, we may assume
that $g$ has 2-power order. Then $g \in P$ and $[g, \oh{2'} G]=1$. Therefore,
$\theta$ is the trivial character, by Burnside's theorem (see Problem 3.16 of
\cite{Isaacs}). Then $B$ is the principal block of $G$, for instance by Fong's
theorem \cite[Thm.~10.20]{Na98}.
\end{proof}
 
\begin{lem}   \label{overunder}
 Let $N\unlhd G$ be finite groups. Suppose that $G/N$ has order not divisible
 by some prime $p$. Let $\sigma\in\Gal(\QQ_{|G|}/\QQ)$ have order a power
 of~$p$.
 \begin{enumerate}[\rm(a)]
  \item If $\chi \in \irrs G$, then all the irreducible constituents of
   $\chi_N$ are $\sigma$-invariant.
  \item Suppose that $p=2$, and $\sigma$ complex-conjugates odd-order roots of
   unity. Suppose that $\theta \in \irrs N$. Then there exists a unique
   $\chi\in\irrs G$ over $\theta$. Furthermore, if $\theta$ is $G$-invariant,
   then $\chi_N=\theta$. Also, $\chi$ belongs to the principal $2$-block of $G$
   if and only if $\theta$ belongs to the principal $2$-block of $N$.
 \end{enumerate}
\end{lem}
 
\begin{proof}
For (a), let $\theta \in \Irr(N)$ be an irreducible constituent of $\chi_N$,
and let $\Omega$ be the set of all the distinct $G$-conjugates of $\theta$.
Notice that $|\Omega|$ divides $|G:N|$, and therefore has size not divisible
by $p$. Since $\chi$ is $\langle\sigma\rangle$-invariant, we have that
$\langle \sigma \rangle$ acts on $\Omega$. Therefore, there is
$\eta \in \Omega$ which is $\sigma$-invariant.
Now, if $g\in G$, then $(\eta^g)^\sigma=(\eta^\sigma)^g=\eta^g$, and so all
elements of $\Omega$ are $\sigma$-invariant.

Next, we prove (b) by induction on $|G:N|$. To prove the first two
assertions, by the Clifford correspondence, we may assume that $\theta$ is
$G$-invariant. If $N=G$, there is nothing to prove. Else, since $G/N$ has odd
order, there is $N\le M \unlhd G$ of prime index in $G$. By induction, there
exists a unique $\sigma$-invariant $\psi \in \irr M$ over~$\theta$.
Furthermore, $\psi_N=\theta$. 
By uniqueness, we have that $\psi$ is $G$-invariant, and therefore $\psi$ (and
$\theta$) extends to $G$, because $G/M$ is cyclic. Let $\tau\in\Irr(G)$ be any
extension of $\theta$. Then $\tau^\sigma=\lambda \tau$, for a unique linear
$\lambda\in\irr{G/N}$ by the Gallagher correspondence \cite[Cor.~6.17]{Isaacs}.
Since $|G/N|$ is odd, there exists a linear $\nu \in \irr{G/N}$ such that
$\nu^2=\lambda$. Now, for $\chi=\nu\tau \in \Irr(G)$ we have
$$\chi^\sigma=(\nu\tau)^\sigma=\nu^{-1}\tau^\sigma=\nu^{-1}\lambda\tau
  =\chi \, .$$
Also, if $\rho \in \irr{G/N}$ and $\rho\chi$ is $\sigma$-invariant, then $\rho$
is $\sigma$-invariant, by the uniqueness in the Gallagher correspondence. 
However, $\rho^\sigma=\bar\rho$, and since $G/N$ has odd order, we obtain that
$\rho=1$, by Burnside's theorem (see Problem 3.16 of \cite{Isaacs}).

If $\chi$ belongs to the principal block of $G$, then $\theta$ belongs to the
principal block of $N$ by \cite[Thm.~9.2]{Na98}. Conversely, assume that
$\theta$ belongs to the principal block of $N$. Since $G/N$ has odd order, it
follows that $\chi$ lies in a block of maximal defect of $G$ (using, for
instance, \cite[Thm.~9.26]{Na98}). Now, the result follows from Lemma \ref{bl}.
\end{proof}
 
From now on, $\sigma\in\Gal(\QQ^\ab/\QQ)$ is the Galois automorphism that
fixes $2$-power roots of unity and complex conjugates odd-order roots of unity.
Of course, $\sigma^2=1$.
\medskip
 
Assuming Proposition~\ref{gunter 3}(a) and~(c), and Theorem~\ref{thm:gunter 2+4}
(on almost-simple groups), the proof of Theorem~B is standard and follows the
proof of \cite[Thm.~4.2]{DNT}.
 
\begin{thm}   \label{dnt}
 Let $G$ be a finite group, and let $P\in \Syl_2(G)$. Then all characters in
 $\irrs G$ have odd degree if and only if $P$ is normal in $G$ and abelian.
\end{thm}
 
\begin{proof}
Suppose that $P \unlhd G$ is abelian. Then all elements of $\Irr(G)$ have
odd degree by It\^o's theorem \cite[Thm.~6.15]{Isaacs}.
 
Now suppose that all characters in $\irrs G$ have odd degree. We claim that if
$P$ is normal in $G$, then $P$ is abelian. To see this, let $\theta\in\irr P$.
By Lemma~\ref{overunder}(b), let $\chi\in\irrs G$ be over $\theta$. By
hypothesis, $\chi$ has odd degree, and by Clifford's theorem we conclude that
$\theta$ is linear.

We prove by induction on $|G|$ that $P$ is normal in $G$. 
We may assume that $P>1$. Arguing by induction, we may assume that $G$ has a
unique minimal normal subgroup $N$ and that $PN \unlhd G$.
In particular, we may assume that $N$ is not a 2-group.
Using Lemma~\ref{overunder}(b), we may assume that $\Oh{2'} G=G$. 
Therefore, we have $PN=G$ and $\cent PN=1$. 

Assume that $N$ has odd order. Let $x \in P$ be an involution, and notice that
there exists $\lambda \in \irr N$ such that $\nu:=\lambda^{-1}\lambda^x\ne1$. 
Then $\nu^x=\bar\nu=\nu^\sigma$, and if $\mu\in \irr{G_\nu}$, is the canonical
extension of $\nu$ to $G_\nu$ \cite[Cor.~6.28]{Isaacs}, then
$\mu^\sigma=\mu^x$, by uniqueness.
It follows that $\mu^G \in \Irr(G)$ is $\sigma$-invariant.
By hypothesis, $\mu^G$ has odd degree, and therefore
$G_\nu=G$, $\bar\nu=\nu$, and this is not possible because $|N|$ is odd.

Hence, we may assume that $N=S^{x_1} \times \cdots \times S^{x_t}$, where
$S\unlhd N$ is non-abelian simple, and $\{x_1,\ldots,x_t\}$ is a complete set
of representatives of the right cosets of $H=\norm GS$ in $G$, with $x_1=1$.
Write $C=\cent GS$. We know that $H/C$ is almost simple with socle $SC/C$,
and with $H/SC$ a 2-group.
We claim that there exists $\tau \in \irrs{H/C}$ of even degree such that $S$
is not in the kernel of $\tau$. Suppose first that $S$ has non-abelian Sylow
2-subgroups. Then the claim follows from Theorem~\ref{thm:gunter 2+4}.
Suppose now that $S$ has abelian Sylow 2-subgroups. If $SC<H$, then the claim
follows from Proposition~\ref{gunter 3}(c).
We are left with the case where $S$ has abelian Sylow 2-subgroups and $H=SC$.
In this case, we take $\tau=\chi_1$ from Proposition \ref{gunter 3}(a). Now, if
$1 \ne \theta \in \Irr(S)$ is an irreducible constituent of $\tau_S$, we have
that $\eta=\theta \times 1_{S^{x_2}} \times \cdots \times  1_{S^{x_t}}$, lies
under $\tau$ and that $G_\eta \sbs H$. By the Clifford correspondence, we have
that $\tau^G \in \Irr(G)$ is $\sigma$-invariant  of even degree and this is a
contradiction. 
\end{proof}

If $G$ is solvable, we can obtain the normality of $P$ by using the main result
of~\cite{Gr}, where it is proved that if $G$ is solvable, $p$ is a prime,
$\sigma$ is a Galois automorphism of $\Gal(\QQ_{|G|}/\QQ)$ of order $p$, and
all $\Irr_\sigma(G)$ have $p'$-degree, then $G$ has a normal Sylow $p$-subgroup.

\section{Proof of Theorem A}   \label{sec:thm A}
In this section we prove Theorem~A, using Proposition~\ref{gunter 3} and
Theorem~\ref{thm:gunter 2+4}. Of course, Theorem~A improves on Brauer's Height
Zero conjecture for 2-principal block (shown in \cite{NT12b}, and more recently
in \cite{MN21} for every prime). We adapt several arguments from \cite{NT12b}
and \cite{MN21} to our present case.

In the first two results, $p$ is an arbitrary prime. For $B$ a $p$-block
of $G$, $Z \unlhd G$ and $\la\in\Irr(Z)$, we set
$\irr{B|\la}:= \irr{G|\la} \cap \irr B$.

\begin{lem}   \label{central}
 Suppose that $G$ is the central product of subgroups $U_i$, $1 \le i \le t$.
 Let $Z=\bigcap U_i$, $\lambda \in \irr Z$, and assume that $Z$ is a $p$-group.
 Suppose that $B_i$ is the principal $p$-block of $U_i$, and that $B$ is the
 principal $p$-block of $G$. Then there is a natural bijection
 $$\irr{B_1|\lambda} \times \cdots \times \irr{B_t|\lambda}
   \rightarrow \irr{B|\lambda} \, .$$
\end{lem}

\begin{proof}
Given $\psi_i \in \irr{U_i|\lambda}$,
we know that there is a unique $\chi \in \irr{G|\lambda}$ such that $\chi$ lies
over every $\psi_i$, and that $(\psi_1, \ldots, \psi_t) \mapsto \chi$ is a
bijection. Furthermore, $\chi(u_1 \cdots u_k)=\psi_1(u_1)\cdots\psi_k(u_k)$ for
$u_i\in U_i$. (See \cite[Thm.~10.7]{Na18}.)
If $\chi$ lies in $B$, then each $\psi_i$ lies in $B_i$ by
\cite[Thm.~9.2]{Na98}, so we only need to prove the converse. We prove it for
$t=2$, which is clearly sufficient.
Since $[U_1, U_2]=1$, $G=U_1U_2$ and $Z \sbs \zent G$ is a $p$-group,
it easily follows that the map $U_1^0 \times U_2^0 \rightarrow G^0$
given by $(x,y) \mapsto xy$ is a bijection, where $G^0$ is the set of
$p$-regular elements of $G$. Now
$$\sum_{(x,y) \in U_1^0 \times U_2^0 } \chi(xy)
  =\left(\sum_{x \in U_1^0}\psi_1(x) \right)
   \left(\sum_{y \in U_2^0}\psi_2(y) \right) \ne 0,$$
using \cite[Thm.~3.19 and Cor.~3.25]{Na98}. Therefore $\chi$ is in the
principal block, again by \cite[Thm.~3.19 and Cor.~3.25]{Na98}.
\end{proof}

The character $\chi$ under the bijection in Lemma \ref{central} is usually
denoted by
$$\chi=\psi_1 \cdots \psi_t \, .$$

We shall use the following surely well-known result (which we used in the
proof of \cite[Thm.~5.4]{MNS}), and whose argument we write down here for the
reader's convenience. Recall that the layer $\bE(G)$ of a finite group $G$
is the product of the components of~$G$, and that
$\bF^*(G)=\bE(G)\bF(G)$ is the generalised Fitting subgroup of $G$.

\begin{lem}   \label{spice}
 Let $G$ be a finite group, let $E=\bE(G)$. Assume that $E\ne 1$ and that
 $\bF(G)$ is central in $G$.
 \begin{enumerate} 
  \item[\rm(a)] We have that
   $$\bigcap_{U} U\cent GU =\bF^*(G), $$
   where $U$ runs over the components of $G$.
  \item[\rm(b)] Suppose that $p$ is a prime dividing $|E|$ but not $|\zent E|$
   and $\oh{p'} G \sbs \zent G$. If $Q\in \Syl_p(E)$, then
   $\bF^*(G)\cent GQ/{\bF^*}(G)$ is solvable.
 \end{enumerate}
\end{lem}

\begin{proof}
Suppose that $\Omega=\{U_1,\ldots,U_n\}$ is the set of the different components
of $G$, so $E=U_1 \cdots U_n$. Write $C_i=\cent G{U_i}$ and $K=\bF^*(G)$.
Let
$$L=\bigcap_{i=1}^n U_iC_i \, .$$
It is well-known that $[U_i,U_j]=1$ (\cite[Thm.~9.4]{Isgt}) and
$[U_i, \bF(G)]=1$ (\cite[Thm.~9.7.c]{Isgt}). In particular, $K \sbs L$.
Suppose that $g \in L$. Write $g=c_iu_i$, where $c_i\in C_i$ and $u_i \in U_i$.
Let $x=u_1 \cdots u_n \in E$. We claim that $e^g=e^x$ for all $e \in E$.
Indeed, if $e=v_1 \cdots v_n$, then
$$e^g=v_1^g \cdots v_n^g=v_1^{c_1u_1} \cdots v_n^{c_nu_n}
  =v_1^{u_1} \cdots v_n^{u_n}=(v_1 \cdots v_n)^{u_1 \cdots u_n}=e^x \, .$$
Then $gx^{-1} \in \cent GE$. Since $\bF(G) \sbs \zent G$, we have
$gx^{-1} \cent G{E\bF(G)}=\cent G{K} \sbs K$ (by using
\cite[Thm.~9.8]{Isgt}). Therefore $g \in K$, showing~(a).

Let $Z=\zent E$. By hypothesis, $p$ does not divide $|Z|$. For~(b), by the
proof of Theorem 9.7 and Lemma 9.6 of \cite{Isgt}, we know that 
$$ E/Z= \prod_{i=1}^n (U_iZ)/Z$$
is a direct product (of non-abelian simple groups).
Let $Q_i=Q \cap {U_iZ}=Q\cap U_i \in \Syl_p(U_i)$. Since $\oh{p'}G\sbs\zent G$,
we have $Q_i\ne 1$. Write $C=\cent GQ$. Let $c \in C$. We claim that~$c$
normalises $U_i$ for all $i$. We have that $c$ permutes the set $\Omega$.
Let $1 \ne y_i \in Q_i$. Since $C \sbs \cent G{Q_i}$, we have $(y_iZ)^c=y_iZ$.
Using that the product is direct, we have that necessarily $c\in\norm G{U_iZ}$. 
Since $(U_iZ)'=U_i$, we conclude that $c \in \norm G{U_i}$. By the Schreier
conjecture, $\norm G{U_i}/U_iC_i$ is solvable. (Recall that $\Out(U_i)$ is a
subgroup of $\Out(U_i/Z_i)$, where $Z_i=Z \cap U_i=\zent{U_i}$.)
Therefore $C/(C \cap U_iC_i)$ is solvable. Hence,
$$C/\Big(C \cap \bigcap_{i=1}^n (U_iC_i) \Big)$$
is solvable. By part~(a), we conclude that $C/(C \cap \bF^*(G))$ is
solvable, as required.
\end{proof}

We are now ready to prove Theorem~A, which we restate:

\begin{thm}   \label{thm:A 2}
 Let $B$ be the principal $2$-block of a finite group $G$, and let
 $P\in\Syl_2(G)$. Then all characters in $\Irr_\sigma(B)$ have odd degree if
 and only if $P$ is abelian.
\end{thm}

\begin{proof}
Write $B_0(G)$ for the principal $2$-block of $G$. 
If $P\in\Syl_2(G)$ is abelian, then all irreducible characters in $B_0(G)$ have
odd degree by \cite{NT12b}.

For the converse, we argue by induction on $|G|$. We may assume that $G$ is not
a 2-group. Assume that $N\unlhd G$ has odd index. We claim that all characters
in $\irrs{B_0(G)}$ have odd degree if and only if those in $\irrs{B_0(N)}$ have
odd degree. But this easily follows from Lemma \ref{overunder}, and the fact
that if $\chi \in \Irr(G)$ and $\theta \in \irr N$ lies under $\chi$,
then $\chi(1)/\theta(1)$ divides $|G:N|$ (\cite[Cor.~11.29]{Isaacs}).
Hence, we may assume that $\Oh{2'} G=G$.
Also, since $\oh{2'} G$ is in the kernel of the characters of $B_0(G)$
(\cite[Thm.~6.10]{Na98}), we may also assume that $\oh{2'} G=1$.

Since $B_0(G/N) \sbs B_0(G)$ if $N$ is normal in $G$, by induction, we deduce
that $G$ has a unique minimal normal subgroup $N$, and that $G/N$ has an
abelian Sylow 2-subgroup. Let $K/N=\oh{2'}{G/N}$.

Now, let $Q=P\cap N\in\Syl_2(N)$ and $M=N\cent GQ \unlhd G$. We know that all
irreducible characters of $G/M$ are in $B_0(G)$, by \cite[Lemma~3.1]{NT12b}.
By Theorem~B, we have that $G/M$ has a normal Sylow $2$-subgroup.
Since $\Oh{2'}G=G$, we conclude that $G/M$ is a 2-group. Hence, $K \sbs M$.

Suppose that $N$ is a 2-group. Then $N$ is elementary abelian (because $N$ is a
minimal normal subgroup) and $M=\cent GN$.
Then $K=N \times \oh{2'}K$, by using the Schur--Zassenhaus theorem. Since
$\oh{2'}G=1$, we have that $K=N$. Then $G/N$ is a direct product of non-abelian
simple groups and a 2-group (by \cite[Thm.~1]{Wa69}).
Write $G/N=L/N \times E/N$, where $E/N$ is a $2$-group, and $L/N$ is the direct
product of non-abelian simple groups $X_i/N$. Since
$L'N=L$ and $G$ has a unique minimal normal subgroup $N$, we conclude that
$L'=L$ is perfect. 
Since $E$ is a 2-group, then $N \cap \zent E>1$, and therefore $E \le M$.
Since $G/M$ is a 2-group, necessarily $M=G$, and $N$ is central. Thus $|N|=2$. 
Now, $[L,E,L]=1=[E,L,L]=1$ and by the three subgroups lemma,
we have that $[E,L]=1$. 
Now, $X_i$ is a normal subgroup of $G$, and since $G$ has a unique minimal
normal subgroup $N$, and $X_i'$ is normal in $G$, we conclude that $X_i$ is
perfect for all $i$. Hence $X_i$ is quasi-simple.
By the same argument as before, $[X_i,X_j]=1$, if $i\ne j$.
By Lemma \ref{3.1}(a), every $X_i$ has non-abelian Sylow 2-subgroups. 
Let $1\ne\lambda\in\Irr(N)$. By Proposition \ref{gunter 3}(b), let
$\chi_i \in \irrs{B_0(X_i)|\lambda}$ with even degree. Also, let
$\mu\in\irr{E|\lambda}$. Then, using Lemma~\ref{central}, we have that
$\chi=\chi_1 \cdots \chi_t \cdot \mu \in\irr{B_0(G)|\lambda}$ has
even degree and is $\sigma$-invariant. This is a contradiction, hence $N$
cannot be a 2-group.

We thus may assume that $N=\bE(G)=\bF^*(G)$ is a minimal normal
subgroup of $G$.

Now, $M/N$ is solvable by Lemma~\ref{spice}(b). Thus $G/N$ is solvable and
since it has abelian Sylow 2-subgroups, it follows that $G/N$ has a normal
$2$-complement~$K/N$ (by using the so called Hall--Higman's Lemma 1.2.3 and the
fact that $\Oh{2'}{G/N}=G/N$). Recall that $K \sbs M$, so $K=N\cent KQ$. 

We claim that we may assume that $G=NP$. Let $\tau\in\irrs{B_0(NP)}$, and let
$\theta\in\irr N$ be under $\tau$. Then $\theta^\sigma=\theta^g$ for some
$g\in P$, by Clifford's theorem, and $\theta$ lies in the principal 2-block
of~$N$. Let $\eta\in\irr{B_0(K)}$ be the unique extension of $\theta$ in the
principal block of $K$ (using Alperin isomorphic blocks, \cite{Al}).
By uniqueness, $\eta^g=\eta^\sigma$.
(To prove this, use that $(\eta^\sigma)^{g^{-1}}$ and $\eta$ are extensions
in the principal block of $K$ of $\theta$, so they coincide.)
Let $I$ be the stabiliser of $\eta$ in $G$. Notice that $J=I\cap PN$ is the
stabiliser of $\theta$ in $PN$ (again using uniqueness in Alperin's theorem).
Let $\mu \in \irr J$ be the Clifford correspondent of $\tau$ over $\theta$.
Again by uniqueness, $\mu^\sigma=\mu^g$.
By the Isaacs restriction correspondence \cite[Lemma~6.8(d)]{Na18}, let
$\rho\in\irr{I|\eta}$ be such that $\rho_J=\mu$. By uniqueness of this
restriction map, we again have $\rho^\sigma=\rho^g$. By the Clifford
correspondence, $\chi=\rho^G=(\rho^\sigma)^G \in \Irr(G)$ is $\sigma$-invariant,
and lies in
the principal $2$-block of $G$ (because $\eta$ does and $G/K$ is a $2$-group).
By hypothesis, we have that $\chi$ has odd degree. Thus $I=G$, $\rho=\chi$, and
$\chi_{PN}=\tau$ has odd degree. By induction, we may assume that $G=PN$,
as claimed.

Write $N=S^{x_1} \times \cdots \times S^{x_t}$ where the $x_i$ are
representatives of the right cosets of $H=\norm GS$ in $G$, and $S$ is
non-abelian simple. Write $C=\cent GS$. Assume that $t>1$ (so that $H<G$).
We claim that there exists $\gamma \in \irrs{H/C}$ in the principal block of
$H/C$ (and therefore of $H$), such that $\gamma_S$ does not contain the trivial
character.  If the Sylow 2-subgroups of $S$ are not abelian, this follows from
Theorem~\ref{thm:gunter 2+4}. Assume that $S$ has abelian Sylow 2-subgroups.
If $SC<H$, then this follows from Proposition~\ref{gunter 3}(c), while if
$SC=H$, it follows from the second part of Proposition~\ref{gunter 3}(a).
Now, if $1\ne\xi\in\Irr(S)$ is under $\gamma$, then the stabiliser
$G_\tau\sbs H$, where $\tau=\xi\times 1\times\cdots\times 1\in\irr N$. By the
Clifford correspondence, we obtain $\gamma^G=\chi \in \irrs {B_0(G)}$ (by
\cite[Cor.~6.2]{Na98} and Brauer's third main theorem \cite[Thm.~6.7]{Na98}),
with even degree (using that $H<G$ and that $|G:H|$ is a 2-power).
This is not possible by hypothesis. We conclude that $t=1$. In this case, the
assertion follows by Theorem~\ref{thm:gunter 2+4}.
\end{proof}

\section{$p$-rational characters}   \label{sec:thm C}
In this final section, we prove Theorem C. We start with the following. Recall
that $\chi \in \irr G$ is \emph{$p$-rational} if the field of values
$\QQ(\chi)=\QQ(\chi(g) \mid g \in G)$ of $\chi$ is contained in a cyclotomic
field $\QQ_m$, with $m$ not divisible by $p$.

\begin{thm}   \label{pratnor}
 Suppose that $\chi\in\Irr(G)$ is $p$-rational, with $\chi^0=\vhi\in \IBr(G)$.
 Let $N \unlhd G$ and let $\theta \in \Irr(N)$ be under $\chi$. Then $\theta$
 is $p$-rational.
\end{thm}

\begin{proof}
We argue by induction on $|G|$. Let $T^*$ be the set of all $g \in G$ such that
$\theta^g=\theta^\sigma$ for some $\sigma \in \Gal(\QQ(\theta)/\QQ)$.
Let $T\sbs T^*$ be the stabiliser of $\theta$ in $G$, and let $\psi\in\irr T$
be the Clifford correspondent of $\chi$ over $\theta$. Let $\eta=\psi^{T^*}$.
It is well-known that $\QQ(\eta)=\QQ(\chi)$ (see for instance, Problem 3.9 of
\cite{Na18}). Since $\chi^0=(\eta^G)^0=(\eta^0)^G=\varphi\in\IBr(G)$, we deduce
that $\eta^0\in \IBr(T^*)$. If $T^*<G$, then we are done by induction. Hence,
we may assume that $T^*=G$. Notice that in this case $T \unlhd G$, since 
the stabilisers $G_\theta=G_{\theta^\sigma}$ coincide.
 
Assume that $\psi$ is not $p$-rational. Then there exists
$\sigma \in \Gal(\QQ_{|G|}/\QQ_{|G|_{p'}})$ such that $\psi^\sigma \ne \psi$.
Now, $\chi^\sigma=\chi$, and thus $\theta^\sigma=\theta^g$ for some $g \in G$,
by Clifford's theorem. Hence $\psi^\sigma=\psi^g$ by uniqueness in the Clifford
correspondence. Now $(\psi^0)^\sigma=\psi^0$ because $\sigma$ fixes $p'$-roots
of unity, and thus $\psi^0=(\psi^\sigma)^0 =(\psi^g)^0=(\psi^0)^g$.
However, $(\psi^0)^G=(\psi^G)^0=\chi^0 \in \IBr(G)$, and therefore we deduce
that $g\in T$ (by Problem 8.3 in \cite{Na98}.) Then $\psi^\sigma=\psi$,
contrary to assumption. Therefore $\psi$ is $p$-rational. Now, $\psi_N=e\theta$
for some $e\ge1$, and we deduce that $\theta$ is $p$-rational.
\end{proof}

\begin{cor}   \label{kernels}
 Suppose that $\chi\in\Irr(G)$ is $p$-rational for some odd prime $p$,
 with $\chi^0=\varphi \in \IBr(G)$. Then $\ker \varphi \le \ker\chi$.
\end{cor}
 
\begin{proof}
We argue by induction on $|G|$. By induction, we may assume that
$\ker\varphi \cap \ker\chi=1$. Let $N=\ker\varphi$. Let $\theta \in \Irr(N)$
be under $\chi$. Also, $\theta^0=e1_N$. Then all $p$-regular elements of $N$
are in $\ker\theta \cap \ker\varphi$, and therefore $N$ is a $p$-group. By
Theorem~\ref{pratnor}  we have that $\theta$ is $p$-rational. Since $N$ is
a $p$-group, we have that $\theta$ is rational valued. Since $p$ is odd,
$\theta=1_N$, as wanted.
\end{proof}
 
We shall prove Theorem~C, based on the following consequence of the
classification of finite simple groups. Here we use $I_G(\theta)=G_\theta$ to
denote the stabiliser in $G$ of the character $\theta$ of a normal subgroup $N$
of $G$, as mentioned before Lemma~\ref{lem:odd}.

\begin{thm}   \label{thm:def 0}
 Suppose that $p$ is an odd prime number. If $S$ is a finite non-abelian simple
 group of order divisible by $p$, then there exists a $p$-rational
 $\eta\in\Irr(S)$ such that $p$ divides $\eta(1)$, $\eta^0 \in\IBr(S)$, and
 $I_{\Aut(S)}(\eta)=I_{\Aut(S)}(\eta^0)$.
\end{thm}

\begin{proof}
Assume that $S$ has a character $\eta\in\Irr(S)$ of $p$-defect zero. Then
$\eta^0\in\IBr(S)$, $p$ divides $\eta(1)$, and clearly the stabilisers of
$\eta$ and $\eta^0$ in $\Aut(S)$ agree. Now it is known that all non-abelian
simple groups have blocks of $p$-defect zero for all $p\ge5$, and for $p=3$ the
only exceptions to this statement are certain alternating groups and the two
sporadic groups $Suz$ and $Co_3$ (see \cite[Cor.~2]{GO96}). For $Suz$ and
$Co_3$ our claim is checked from the known decomposition matrices in \cite{GAP}:
$Suz$ has a unique irreducible character $\eta$ with $\eta(1)=15795$, and that
remains irreducible modulo~3, and $Co_3$ has a unique irreducible character
$\eta$ with $\eta(1)=91125$ in a block of defect~1 that remains irreducible
modulo~3.   \par
The alternating group $\fA_5$ has two characters of 3-defect zero.
For $\fA_n$ with $n\ge6$ let $\chi$ be the irreducible character of $\fS_n$
labelled by the partition $\la=(n-1,1)$ if $n\equiv1\pmod3$, $\la=(n-2,1^2)$ if
$n\equiv2\pmod3$ and $\la=(n-2,2)$ if $n\equiv0\pmod3$. Then by the hook
formula, $\chi(1)$ is $n-1$, $(n-1)(n-2)/2$, $n(n-3)/2$ respectively, so
divisible by~3.
Furthermore, since $\la$ is not self-conjugate, $\chi$ restricts to an
irreducible character $\eta$ of $\fA_n$ that is rational. Finally, $\la$ is
a 3-JM partition in each case, so $\eta^0\in\IBr(\fA_n)$ by
\cite[Thm.~3.1]{F16}. This completes the proof.
\end{proof}

We also need this lemma.

\begin{lem}   \label{gab}
 Suppose that $p$ is an odd prime. Let $G$ be a finite group,
 $N\unlhd G$, and let $\theta \in \Irr(N)$ be $p$-rational such that
 $\theta^0\in\IBr(N)$ and
 $I_G(\theta)=I_G(\theta^0)$. Assume that $G/N$ is $p$-solvable. Then there
 exists a $p$-rational $\chi\in\irr{G|\theta}$ such that $\chi^0\in\IBr(G)$.
\end{lem}

\begin{proof}
Let $T=I_G(\theta)=I_G(\theta^0)$, and assume that $T<G$. Arguing by induction
on $|G:N|$, there exists a $p$-rational $\psi\in\irr{T|\theta}$ such that
$\psi^0 \in \IBr(T)$. Now, $\psi^0 \in \IBr(T|\theta^0)$. By the Clifford
correspondence for Brauer characters, we have $(\psi^0)^G \in \IBr(G)$. Now,
$(\psi^0)^G=(\psi^G)^0$, and therefore $(\psi^G)^0$ is irreducible (and
$\psi^G \in \irr G$ is $p$-rational). So we may assume that $T=G$.
By \cite[Thm.~4.5]{MNS}, we may then assume that $N$ is a central $p'$-subgroup.
In particular, $G$ is $p$-solvable. 
Now, let $\varphi \in \IBr(G)$ lying over $\theta$, and let $\chi \in \Irr(G)$
be the $p$-rational lift of $\chi$ (by the main result of \cite{Is74}). Then
$\chi$ is over $\theta$ and we are done. 
\end{proof}

This is Theorem~C:

\begin{thm}   \label{normal}
 Let $G$ be a finite group, and let $p$ be an odd prime number.
 Let $P \in \Syl_p(G)$. Then $P\unlhd G$ if and only if every $p$-rational
 $\chi\in\Irr(G)$ with $\chi^0\in\IBr(G)$ has degree not divisible by $p$.
\end{thm}

\begin{proof}
Assume that $P \unlhd G$, and let $\chi\in\Irr(G)$ be $p$-rational with
$\chi^0\in\IBr(G)$. We know that $P\sbs \ker{\chi^0}$
(by \cite[Lemma~2.32]{Na98}). Let $\theta\in\Irr(P)$ be an irreducible
constituent of~$\chi_P$. By Corollary~\ref{kernels}, we have $\theta=1_P$.
Therefore, $\chi \in \Irr(G/P)$, and $p$ does not divide $\chi(1)$.

We prove the converse by induction on $|G|$. Arguing by induction, we have that
$G$ has a unique minimal normal subgroup $N$, and that $G/N$ has a normal Sylow
$p$-subgroup $PN/N$. If $N$ is a $p$-group, then $P \unlhd G$,
and we are done. If $N$ is a $p'$-group, then $G$ is $p$-solvable
and the theorem follows. (For instance, using that every $\varphi \in \IBr(G)$
has a $p$-rational lift, by \cite{Is74}, and \cite[Thm.~13.1(c)]{MW}.)

So we assume that $N=S^{x_1} \times \cdots \times S^{x_t}$, where $S$ is a
non-abelian simple group of order divisible by $p$, and $\{x_1,\ldots,x_t\}$ is
a complete set of representatives of the right cosets of $H=\norm GS$ in $G$. 
By Theorem \ref{thm:def 0}, there exists a $p$-rational $\eta \in \Irr(S)$ such
that $\eta^0 \in \IBr(S)$, $p$ divides $\eta(1)$, and
$I_{\Aut(S)}(\eta)=I_{\Aut(S)}(\eta^0)$. Let
$$\theta=\eta^{x_1} \times \cdots \times \eta^{x_t} \in \Irr(N) \, .$$
Then $\theta \in \Irr(N)$ and
$$\theta^0=(\eta^0)^{x_1} \times \cdots \times (\eta^0)^{x_t}\in\IBr(N) \, .$$
We claim that $I_G(\theta)=I_G(\theta^0)$. Suppose $x \in G$ fixes $\theta^0$.
We have
$$S^{x_ix^{-1}}=S^{x_{\sigma(i)}}$$
for a permutation $\sigma \in \fS_t$. Thus $x_{\sigma(i)} x=h_ix_i$ for some
$h_i \in \norm GS$. Then it is easy to check that
$$\left((\eta^0)^{x_1} \times \cdots \times (\eta^0)^{x_t}\right)^x
  =(\eta^0)^{x_{\sigma(1)}x} \times \cdots \times(\eta^0)^{x_{\sigma(t)}x}
  =(\eta^0)^{h_1x_1} \times \cdots \times(\eta^0)^{h_tx_t} \, .$$
We conclude that
$$(\eta^0)^{h_i}=\eta^0$$
for all $i$. Now, $\norm GS/S\cent GS \le \Aut(S)$, and we conclude that
$\eta^{h_i}=\eta$. Then
$$\theta^x=\left(\eta^{x_1} \times \cdots \times \eta^{x_t}\right)^x
  =\eta^{x_{\sigma(1)}x} \times \cdots \times\eta^{x_{\sigma(t)}x}
  =\eta^{h_1x_1} \times \cdots \times\eta^{h_tx_t}
  =\eta^{x_1} \times \cdots \times\eta^{x_t}=\theta \, .$$
By Lemma \ref{gab}, there exists a $p$-rational $\chi \in \irr{G|\theta}$ such
that $\chi^0 \in \IBr(G)$. Since $p$ divides $\theta(1)$ and $\chi(1)$ is not
divisible by $p$, by hypothesis, we get a contradiction.
\end{proof}

The following gives an example where $p$-rational lifts of an irreducible Brauer
character are not necessarily unique. Furthermore, it is an example of an
irreducible $p$-rational character lifting an irreducible Brauer character
$\chi$ which is not \emph{automorphic} (see Definition~5.3 of \cite{Is74}).
These are characters $\chi\in\Irr(G)$ such that $\chi^0\in\IBr(G)$ and
$\chi^a=\chi$ whenever $a \in \Aut(G)$ fixes $\chi^0$.

\begin{exmp}   \label{exmp:An}
 Let $p\ge3$ be a prime and $G=\fA_n$ the alternating group of degree~$n:=p^2$.
 Let $\chi$ be the irreducible character of $\fS_n$ labelled by the hook
 $\la=((n+1)/2,1^{(n-1)/2})$. Since $\la$ is self-conjugate, $\chi$ splits
 into two distinct characters $\eta_1,\eta_2$ upon restriction to $\fA_n$. By
 \cite[Thm.~2.5.13]{JK} the $\eta_i$ only differ in their value on elements
 whose cycle type is made up of the hook lengths on the main diagonal of $\la$,
 hence on $n$-cycles, and if the values there are irrational, then they
 involve $\sqrt{(-1)^{(n-1)/2}n}$, which in our situation is an integer. Thus,
 both $\eta_i$ are rational.   \par
 Since $\eta_1$ and $\eta_2$ only differ on $n$-cycles, that is, elements of
 order~$p^2$, $\eta_1^0=\eta_2^0$. Now the hook length of $\la$ in the
 $(1,1)$-node is $n=p^2$, hence has positive $p$-valuation. This means that
 $\la$ is an $R$-partition of type I in the sense of \cite[\S2.6]{F16}. Then
 $\eta_1^0=\eta_2^0$ is an irreducible Brauer character by \cite[Thm.~3.1]{F16}.
\end{exmp}


\end{document}